\documentclass[a4paper]{article}
\usepackage{amsfonts}
\usepackage{amsthm}
\usepackage{comment}
\usepackage{amssymb}
\usepackage{amsmath}
\usepackage{stmaryrd}
\usepackage{geometry}
\usepackage{graphicx} 
\usepackage{hyperref}
\hypersetup{
	colorlinks,
	linkcolor={blue!50!black},
	citecolor={blue!50!black},
	urlcolor={blue!50!black}
}
\usepackage{cleveref}
\usepackage{tikz}
\usepackage{tasks}
\usepackage{todonotes}
\usetikzlibrary{arrows,decorations.pathmorphing,backgrounds,positioning,fit,petri,knots,patterns}

\newcommand{\DP}{\operatorname{DP}}
\newcommand{\coDP}{\operatorname{co-DP}}
\newcommand{\F}{\mathcal{F}}
\newcommand{\Le}{\mathcal{L}}
\newcommand{\coLe}{\mathcal{L}^c}
\newcommand{\Z}{\mathbb{Z}}
\newcommand{\Zd}{\mathbb{Z}^d}
\newcommand{\N}{\mathbb{N}}
\newcommand{\A}{\mathcal{A}}

\newcommand{\B}{\mathcal{B}}

\newtheorem{theorem}{Theorem}[section]
\newtheorem{lemma}[theorem]{Lemma}

\newtheorem{proposition}[theorem]{Proposition}
\newtheorem{corollary}[theorem]{Corollary}

\theoremstyle{definition}
\newtheorem{definition}[theorem]{Definition}
\newtheorem{example}[theorem]{Example}
\newtheorem{remark}[theorem]{Remark}

\newtheorem{question}[theorem]{Question}

\title{Parametrized complexity of relations between multidimensional subshifts}
\author{Nicanor Carrasco-Vargas\thanks{Jagiellonian University, Krakow, Poland; \url{nicanor.vargas@uj.edu.pl}}\\ Benjamin Hellouin de Menibus\thanks{Université Paris-Saclay, CNRS, Laboratoire Interdisciplinaire des Sciences du Numérique, 91400 Orsay, France;  \url{hellouin@lisn.fr}}\\Rémi Pallen\thanks{Université Paris-Saclay, ENS Paris-Saclay, 91190 Gif-sur-Yvette, France; \url{remi.pallen@loria.fr}}}
\date{}

\begin{document}

\maketitle
\begin{abstract}
We study the parametrized complexity of fundamental relations between multidimensional subshifts, such as equality, conjugacy, inclusion, and embedding, for subshifts of finite type (SFTs) and effective subshifts. We build on previous work of E. Jeandel and P. Vanier on the complexity of these relations as two-input problems, by fixing one subshift as parameter and taking the other subshift as input. We study the impact of various dynamical properties related to periodicity, minimality, finite type, etc. on the computational properties of the parameter subshift, which reveals interesting differences and asymmetries.

Among other notable results, we find choices of parameter that reach the maximum difficulty for each problem; we find nontrivial decidable problems for multidimensional SFT, where most properties are undecidable; and we find connections with recent work relating having computable language and being minimal for some property, showing in particular that this property may not always be chosen conjugacy-invariant.
 \end{abstract}
\thispagestyle{empty}
\clearpage
\setcounter{page}{1}

\section*{Acknowledgements}
The authors are grateful to Sebasti\'an Barbieri for pointing out that \cite{barbieri2024zero} provides a proof for Theorem~\ref{prop:alldegree} and to Mathieu Hoyrup for the proof of Theorem \ref{thm:upper-bound-incl}.

\section{Introduction}
A subshift or shift space is a collection of colorings of the $d$-dimensional grid $\mathbb{Z}^d$ ($d\geq 1$) by a finite set of colors and avoiding a set of forbidden patterns. We are interested in subshifts of finite type (SFT) and effective subshifts, which can be defined by a set of forbidden patterns which is finite or recursively enumerable, respectively. 

Subshifts, in particular SFTs, have received a lot of interest as a source of a great variety of natural undecidable problems, starting from the Domino problem. In \cite{jeandel_hardness_2015}, Jeandel and Vanier described the computational difficult of a number of relations between subshifts: equality, conjugacy, inclusion, and embedding. We are interested in the parametrized version of the same problems, that is when one subshift is fixed and the other is given as input.

\begin{quote}
    \textbf{Conventions.} We denote by $Y$ the parameter subshift, which is fixed beforehand, and by $X$ the input subshift, which ranges over the class of SFTs or the class of effective subshifts. The integer $d$ always refers to the dimension. 
\end{quote}
In this work we will be interested in the following problems:
\begin{tasks}[style=itemize](6)
    \task {$X=Y$} 
    \task {$X\simeq Y$} 
    \task {$X\subseteq Y$} 
    \task {$X\hookrightarrow Y$} 
    \task {$Y\subseteq X$} 
    \task {$Y\hookrightarrow X$}
\end{tasks}
The symbol $\simeq$ indicates the existence of a conjugacy (continuous shift-equivariant bijection), and the symbol $\hookrightarrow$ indicates the existence of an embedding (continuous shift-equivariant injection). 

Let us explain our language by means of an example. Let $Y$ be the fullshift $\{0,1\}^{\Z^2}$. By ``the problem $X\simeq Y$ on SFT inputs'' we mean the problem of, on input an SFT $X$, determining whether the relation $X\simeq Y$ holds. An SFT input is specified by an alphabet and a finite collection of forbidden patterns, while an effective subshift input is specified by an alphabet and a Turing machine which recursively enumerates forbidden patterns.    
 
One motivation for this approach is to understand the relationship between the computational difficulty of these problems, seen as properties of $Y$, and its dynamical and combinatorial properties. While the two-input version of each problem has a single difficulty, 
some choices of parameter $Y$ make the problem as difficult as possible (as the two-input version) or, conversely, make the problem easier, and we manage to describe and sometimes characterize which properties of $Y$ impact which problem. Strikingly, these properties vary a lot depending on the considered problem: periodicity, minimality, finiteness, being of finite type, etc. For example, we find a connection between our results on $X\simeq Y$ and the recent work \cite{amir_minimality_2025}, which links computability and minimality properties of $Y$ (see Theorem \ref{sec:conjugacy}). This new finer perspective provides more information on the relationship between dynamical and computational properties of shifts.

A second motivation to study these parametrized problems is their connection to Rice-like theorems for properties of SFTs \cite{cervelle_tilings_2004,delvenne_quasi-periodic_2004,lafitte_computability_2008,carrasco-vargas_rice_2025}, which prove that most relevant properties of SFTs are undecidable, as expressed in the well-known metaphor ``swamp of undecidability'' from \cite{lind_multi-dimensional_2004}. While most properties that we consider can be easily shown to be undecidable thanks to these results, we do find some nontrivial decidable problems for SFT inputs (such as $Y\subseteq X$ and $Y\hookrightarrow X$ when $Y$ is finite), highlighting limits of these generic Rice-type theorems. Providing a complete characterization of decidable properties for SFT seems a worthy long-term goal.

Our main contributions can be summarized as follows.
\begin{description}
\item Section \ref{sec:Y-contained-in-X}: for an arbitrary $Y$, the problem $Y\subset X$ with SFT inputs is reducible to $\Le^c(Y)$, so it is decidable exactly when $Y$ has computable language. 
\item Section \ref{sec:Y-embeds-X}: when $Y$ is SFT, the problem $Y\hookrightarrow X$ with SFT inputs is always in $\Sigma_1^0$. We provide some sufficient conditions on $Y$ for decidability (being finite) and for undecidability (containing an almost periodic configuration). Surprisingly, this problem may sometimes be easier than $Y\subset X$ with SFT inputs.
\item Section \ref{sec:X-contained-in-Y-and-X-embeds-Y}: when $Y$ is effective, the problems $X\subset Y$ and $X\hookrightarrow Y$ are in $\Pi^0_2$ and $\Sigma^0_3$, respectively, and we provide examples of $Y$ reaching these upper bounds. Both problems go down to being $\Sigma^0_1$-complete when $Y$ contains finitely many subsystems of finite type.
\item Section \ref{sec:equality}: when $Y$ is effective, the problem $X=Y$ with effective inputs has complexity between $D(\Sigma_1^0)$ and $\Pi_2^0$. We show that the complexity reaches its upper bound if $Y$ is not an SFT, and its lower bound if $Y$ is an SFT with computable language.
\item Section \ref{sec:conjugacy}: when $Y$ is SFT, the problem $X\simeq Y$ with effective inputs has complexity between $D(\Sigma_1^0)$ and $\Sigma_3^0$. We find that the upper bound is reached on some examples, and the lower bound is reached if $Y$ is minimal for a conjugacy invariant $\Pi_1^0$ property (a result related with \cite{amir_minimality_2025}).
\end{description}

We have chosen to omit sofic subshifts from our results for space reasons, though we expect that all results on effective subshifts on dimension $d=1$ are also valid for sofic subshifts on dimension $d\geq 2$. This follows from classical simulation results \cite{aubrun2013simulation,durand_effective_2010} which are explained in detail in \cite[Section 1.3]{jeandel_hardness_2015}. We also expect that many of our arguments can be extended to subshifts over more general finitely generated groups with decidable word problem (and in some cases, undecidable domino problem), but we have chosen to restrict our study to $\mathbb{Z}^d$ as more tools are available.

\section{Preliminaries
}\label{sec:preliminaries}

\subsection{Subshifts}\label{def:subshifts}

Let $d\geq 1$ and let $\A$ be a finite alphabet. We endow $\A$ with the discrete topology and $\A^{\Z^d}=\prod_{\Z^d}\A=\{x\colon \Z^d\to \A\}$ with the product topology. The elements in $\A^{\Zd}$ are called \textbf{configurations}. We also consider the continuous left action $\Z^d \curvearrowright \A^{\Zd}$ by \textbf{shift translations}. Given $u\in\Z^d$ and $x\in \A^{\Zd}$ we define its shift or translation $\sigma^u(x)\in \A^{\Zd}$ by $\sigma^u(x)(v)=x(v+u)$, $v\in\Z^d$.

A \textbf{pattern} is a function $p\colon S\to \A$, where $S\subset\Zd$ is finite. We say that $p$ \textbf{appears} or \textbf{occurs} in a configuration $x$ when there is $u\in\Z^d$ such that $\sigma^u(x)|_S=p$. 
 A configuration $x$ is \textbf{strongly periodic} if its orbit $\{\sigma^u(x) : u\in\Z^d\}$ is finite; intuitively, $x$ can be described by a finite pattern repeated periodically following $d$ linearly independant period vectors.
 
A \textbf{subshift} is a (possibly empty) subset of $\A^{\Zd}$ which is topologically closed and invariant by shift translations. The subshift $\A^{\Zd}$ is called the \textbf{fullshift} on alphabet $\A$.
Alternatively, any subshift can be defined by forbidding a set of patterns $\F$: we define
\[X_{\mathcal F}=\{x\in \A^{\Zd} : \text{no pattern from $\mathcal F$ appears in $x$}\}.\]
A subshift $X$ is said to have \textbf{finite type} (\textbf{SFT} for short) if there is a finite set $\F$ such that $X=X_{\F}$. A subshift is \textbf{effective} when $X=X_{\F}$ for a $\Sigma_1^0$ set $\F$. The class of SFTs is properly contained in the class of effective subshifts. In the context of problems, SFTs are presented by their $\F$, and effective subshifts by the Turing machine which enumerates their $\F$.

A subshift is \textbf{minimal} if it contains no nonempty subshifts other than itself. 

The \textbf{language} $\Le (X)$ of a subshift $X$ is the set of all patterns that appear in some configuration in $X$, and $\coLe(X)$ is its complement. We remark that $X$ is effective exactly when $\Le(X)$ is a $\Pi_1^0$ set. 
We also define 
$\Le_n(X) \subseteq \Le(X)$ as the subset of patterns with support $\{0,\dots,n-1\}^d$.

\subsection{Embeddings and conjugacies}
A function $F$ from a subshift $X$ to a subshift $Y$ is a \textbf{morphism} when it is continuous and commutes with shift translations on $X$ and $Y$. We call it  \textbf{conjugacy} when it is bijective, \textbf{embedding} when it is injective, and \textbf{factor map} when it is surjective. A property of subshifts is called  a \textbf{dynamical property} when it is preserved by conjugacy. 

By the Curtis-Hedlund-Lyndon theorem, $F\colon X\to Y$ is a morphism if and only if $F$ is a \textbf{sliding block code}: there is a finite $D\subset \Zd$ and a \textbf{local function} $f\colon \A^{D}\to \B$ from the alphabet of $X$ to the one of $Y$ such that
\begin{equation}\label{local-rule}
\forall u\in \Z^d,\ F(x)(u)=f((\sigma^u(x))|_{D}).
\end{equation}
We extend $F$ to map patterns to patterns following Equation \ref{local-rule}: if $p$ has support $S$, then $F(p)$ has support $\{u\in\Z^d : u+D\subseteq S\}$.
The \textbf{radius} of $F$ is the smallest $r\in\N$ such that $F$ can be defined with a local rule with support $\{-r,\dots,r\}^d$.

The \textbf{topological entropy} of a subshift $X$ is defined as:
\[h_{top}(X)=\lim_{n\to\infty}\frac{1}{n^d}\log_2 |\Le_n(X)|\]
An important property is that $h_{top}(X)$ is a conjugacy invariant, and $h_{top}(X)\leq h_{top}(Y)$ whenever $X\subseteq Y$ or when $X$ embeds into $Y$. 
We say that $X$ is \textbf{entropy-minimal} when it has positive entropy and any proper subsystem has strictly lower entropy.

\subsection{Computability}

\begin{definition}
    If \textbf{P} is a problem, \textbf{coP} is the same problem where positive and negative outputs are switched. Furthermore, we say that $\textbf{P}$ is trivial when the answer is the same for every input (either ``yes'' or ``no'').
\end{definition}
The definition of trivial property is similar. 

Given a Turing machine $M$ and $n\in\N=\{0,1,\dots\}$ we write  $M(n)\downarrow$ if $M$ halts on input $n$, and $M(n)\uparrow$ otherwise. If $M(n)\downarrow$ then we denote by $M(n)$ the output of $M$ with input $n$. 

Our hardness proofs use reductions to the following classical problems:

\begin{definition}
    The problem \textbf{Halt} is the algorithmic problem of determining if a Turing machine halts on the empty input (denoted $M(\varepsilon)\downarrow$).
\end{definition}

\begin{definition}
    The problem \textbf{Total} is the algorithmic problem of determining if a Turing machine halts on every input.
\end{definition}

\begin{definition}\label{COFhard}
    The problem \textbf{COF} is the algorithmic problem of determining if a Turing machine has a cofinite language, i.e. whether it does not halt on only finitely many inputs.
\end{definition}

\begin{definition}The \textbf{domino problem} $\DP(\Zd)$ is the algorithmic problem of determining whether an SFT on $\Z^d$ is nonempty.
\end{definition}

\begin{proposition}[Theorems 2.4.2, 4.3.2 and 4.3.3 in \cite{soare_turing_2016}]\label{classicalpb}
    \textbf{Halt} is $\Sigma^0_1$-complete.
    \textbf{COF} is $\Sigma^0_3$-complete.
    \textbf{Total} is $\Pi^0_2$-complete.
\end{proposition}
Some problems we consider lie outside the standard arithmetical hierarchy. Let us recall that a subset of the natural numbers belongs to $D(\Sigma_1^0)$ when it can be written as the difference of two $\Sigma_1^0$ sets. 
\begin{proposition}
$\DP(\Zd)$ is decidable for $d=1$ \cite{lind_introduction_1995}, and $\Pi_1^0$-complete for $d\geq 2$ \cite{berger_undecidability_1966}. 
\end{proposition}

SFT in dimension $d\geq 2$ are able to embed computation in various ways; in our proofs, we use either the SFT given in Figure \ref{fig:computationtiles} or the Robinson SFT \cite{robinson_undecidability_1971}. 

\begin{figure}[h]
    \centering
\begin{tikzpicture}[scale=1.5]
\fill[gray!40] (0,0) -- (1,0) -- (1,1) -- (0,1) -- cycle;

\draw (1.5,0) -- (1.5,1) -- (2.5,1) -- (2.5,0) -- cycle;
\fill[gray!40] (1.5,0) -- (1.5,1)  -- (1.75,1) -- (1.75,0) -- cycle;

\draw (3,0) -- (3,1) -- (4,1) -- (4,0) -- cycle;
\fill[gray!40] (3,0) -- (3,0.25)  -- (4,0.25) -- (4,0) -- cycle;
\draw (3.5,1) node[above] {\tiny $a$};
\draw (3.5,0.5) node {$a$};

\draw (0,-1.5) -- (0,-0.5) -- (1,-0.5) -- (1,-1.5) -- cycle;
\fill[gray!40] (0,-1.5) -- (0,-1.25)  -- (1,-1.25) -- (1,-1.5) -- cycle;
\draw (0,-1) -- (0.25,-1) -- (0.25,-0.5);
\draw (0,-1) node[left] {\tiny $q_0$};
\draw (0.25,-0.5) node[above] {\tiny $q_0,a$};
\draw (0.5,-1) node {$a$};


\draw (3,-1.5) -- (4,-1.5) -- (4,-0.5) -- (3,-0.5) -- cycle;
\draw (3.5,-1) node {$a$};
\draw (3.5,-0.5) node[above] {\tiny $a$};
\draw (3.5,-1.5) node[below] {\tiny $a$};

\draw (1.5,-1.5) -- (2.5,-1.5) -- (2.5,-0.5) -- (1.5,-0.5) -- cycle;
\draw (2,-1) node {$a$};
\draw (2.5,-0.75) -- (1.75,-0.75) -- (1.75,-0.5);
\draw (1.75,-0.5) node[above] {\tiny $q,a$};
\draw (2.5,-0.75) node[right] {\tiny $q,\leftarrow$};
\draw (2,-1.5) node[below] {\tiny $a$};

\draw (4.5,0) -- (5.5,0) -- (5.5,1) -- (4.5,1) -- cycle;
\draw (5,0.5) node {$a$};
\draw (4.5,0.75) -- (4.75,0.75) -- (4.75,1);
\draw (4.75,1) node[above] {\tiny $q,a$};
\draw (4.5,0.75) node[left] {\tiny $q,\rightarrow$};
\draw (5,0) node[below] {\tiny $a$};

\draw (6,-1.5) -- (7,-1.5) -- (7,-0.5) -- (6,-0.5) -- cycle;
\draw (6.5,-1) node {$a$};
\draw (6.25,-1) circle (0.1);
\draw (6.25,-1) node {\small $q$};
\draw[->] (6.25,-1.5) -- (6.25,-1.1);
\draw (6.25,-0.9) -- (6.25,-0.75) -- (7,-0.75);
\draw (6.5,-0.5) node[above] {\tiny $a'$};
\draw (6.25,-1.5) node[below] {\tiny $q,a$};
\draw (7,-0.75) node[right] {\tiny $q',\rightarrow$};

\draw (6.7,-1.8) node {\tiny if $\delta(q,a)=(q',a',\rightarrow)$};

\draw (6,0) -- (6,1) -- (7,1) -- (7,0) -- cycle;
\fill[gray!40] (6,0) -- (6,1) -- (6.25,1) -- (6.25,0.25) -- (7,0.25) -- (7,0) -- cycle;
\draw (6.5,0.5) -- (7,0.5);
\draw (7,0.5) node[right] {\tiny $q_0$};
\draw (6.5,0) node[below] {\tiny \textit{Initial tile}};

\draw (4.5,-1.5) -- (5.5,-1.5) -- (5.5,-0.5) -- (4.5,-0.5) -- cycle;
\draw (5,-1) node {$a$};
\draw (4.75,-1) circle (0.1);
\draw (4.75,-1) node {\small $q$};
\draw[->] (4.75,-1.5) -- (4.75,-1.1);
\draw (4.75,-0.9) -- (4.75,-0.75) -- (4.5,-0.75);
\draw (5,-0.5) node[above] {\tiny $a'$};
\draw (4.75,-1.5) node[below] {\tiny $q,a$};
\draw (4.5,-0.75) node[left] {\tiny $q',\leftarrow$};

\draw (4.8,-1.8) node {\tiny where $\delta(q,a)=(q',a',\leftarrow)$};
\draw (1.55,-2.2) node[right] {\small One copy for each $a \in \Sigma\cup\{B\}, \ q \in Q$.};
\end{tikzpicture}
    \caption{An SFT that simulates a Turing machine $M$ on a free input, where $M$ has alphabet $\Sigma$ with blank symbol $B$, set of states $Q$, transition function $\delta$ and initial state $q_0$. The tiles are the elements of the alphabet, and adjacent edges with non-matching labels are forbidden. 
    It is straightforward to modify the tiles to enforce an empty input or any fixed finite input.
    This procedure of associating an SFT to a Turing machine originates in Hao Wang's proof that the origin constrained version of the domino problem is undecidable \cite{wang_proving_1961}, and the reader is referred to \cite{jeandel_undecidability_2020} for a more detailed exposition of the same construction. }
    \label{fig:computationtiles}
\end{figure}

\section{General upper and lower bounds}\label{sec:general-results}
The goal of this section is to review general upper and lower bounds for the parametrized problems of interest. By identifying simple cases, where the complexity of the parametrized problem can be derived from general facts, we also find more interesting cases that we handle in future sections. 

We start by reviewing hardness results for SFT inputs. From \cite{carrasco-vargas_rice_2025} we have some general tools to prove hardness results for dynamical properties of SFTs in dimension $d\geq 2$.
\begin{definition}\label{def:berger}
    A property $\mathcal P$ of SFTs is a \emph{Berger} property if there are two SFTs $X_{-}$ and $X_{+}$ such that 
    \begin{enumerate}
    \item $X_+$ satisfies $\mathcal P$;
    \item Every SFT $X$ for which there is a factor map from $X$ to $X_-$ does not satisfy $\mathcal P$;
    \item There is a morphism $X_+ \to X_-$.
    \end{enumerate}
\end{definition}
\begin{theorem}\label{adyan-rabin-theorem-for-sfts}
Let $d\geq 2$. 
    Every Berger property of $\mathbb{Z}^d$-SFTs is $\Sigma_1^0$-hard. 
\end{theorem}

This result was stated in \cite{carrasco-vargas_rice_2025} under the assumption that the property is preserved by conjugacy, but this assumption is in fact not necessary. We provide a proof of this fact in Appendix \ref{appendix-rice-theorems}. Theorem \ref{adyan-rabin-theorem-for-sfts} provides a lower bound for four out of the six parametrized problems mentioned in the introduction. Remark that the next result puts no assumptions on the parameter subshift (such as being SFT or effective). 
\begin{corollary}\label{prop:some-berger-properties}
    Let $d\geq 2$. Let $Y$ be a  $\Z^d$-subshift. The following problems with SFT inputs are either $\Sigma_1^0$-hard, or trivial.
    \begin{tasks}[style=itemize](6)
        \task $X=Y$
        \task $X\simeq Y$
        \task $X\subseteq Y$
        \task $X\hookrightarrow Y$
    \end{tasks}
\end{corollary}
    \begin{proof}
    Let $\A$ be the alphabet for $Y$. For $X\subseteq Y$ and $X\hookrightarrow Y$ we choose  $X_+ = \emptyset$ and $X_-$ the full shift on $|\A|+1$ symbols. The condition 2 from the definition of Berger properties is ensured by the fact that if there exists a factor map from $X$ to $X_-$ and $X$ is not conjugate to $X_-$, then $h_{top}(X)>h_{top}(X_-)$. For $X=Y$ and $X\simeq Y$ when $Y$ is not SFT, note that no SFT $X$ satisfies the property, so the property is trivial. If $Y$ is an SFT, then we choose $X_{+}=Y$ and $X_-=Y \times 2^{\Z^2}$. Here the second condition of Berger properties is ensured by the fact that $h_{top}(X_-)=h_{top}(Y \times 2^{\Z^2})=h_{top}(Y)+1$.
\end{proof}
\begin{remark}
Let $Y$ be a subshift which is not of finite type, and consider the problem $X\simeq Y$ with SFT inputs. In this case the answer is ``no'' on every input, so the problem is decidable. This often occurs when the parameter subshift does not belong to the class of input subshifts, so it is important to exclude these uninteresting cases from undecidability results. 
\end{remark}

As mentioned in the introduction, hardness results for $Y\subseteq X$ and $Y\hookrightarrow X$ with SFT inputs cannot be obtained with the same direct arguments. We study these problems in Section \ref{sec:Y-contained-in-X} and Section \ref{sec:Y-embeds-X}.

We now review upper bounds for SFT inputs. In dimension $d\geq 2$, 4 out of the 6 problems mentioned in the introduction admit a straightforward solution when we restrict both input and parameter to the class of SFTs. This generalizes \cite[Theorem 2.4]{jeandel_hardness_2015}. 

\begin{proposition}\label{upper-bounds-two-SFTs}\label{upper-bound-inclusion-uniform-sfts}\label{conjugacy-two-SFTs} 
Let $d\geq 2$, and let $Y$ be a nonempty $\Z^d$-SFT. The following problems are $\Sigma_1^0$-complete for SFT inputs:
\begin{tasks}[style=itemize](6)
    \task $X=Y$
    \task $X\simeq Y$
    \task $X\subseteq Y$
    \task $X\hookrightarrow Y$
\end{tasks} 
The same is true when both $X$ and $Y$ are SFTs and given as inputs.
\end{proposition}

To prove Proposition \ref{upper-bounds-two-SFTs} we need the following result. 
\begin{lemma}\label{upper-bound-morphisms}\label{upper-bound-embeddings}\label{upper-bound-inclusion}
    The following problems are $\Sigma_1^0$.
    Given as input an effective subshift $X\subseteq \A^{\Zd}$, an SFT $Y\subseteq \B^{\Zd}$, and (for 2. and 3.) a local function defining a morphism $F\colon \A^{\Zd}\to \B^{\Zd}$,
    \begin{enumerate}
    \item Determining whether $X\subseteq Y$;
    \item Determining whether $F(X)\subseteq Y$ (equivalently, $F$ is a morphism from $X$ to $Y$);
    \item Determining whether $F$ is an embedding from $X$ to $Y$.
    \end{enumerate}
\end{lemma}

This lemma is a slight generalization of Lemmas 2.1 and 4.2 in \cite{jeandel_hardness_2015}, which are stated for two SFTs. Since the proof is essentially the same, we defer it to the appendix.
 \begin{proof}[Proof of Proposition \ref{upper-bounds-two-SFTs}]
 The upper bound for $=$ and $\subseteq$ follows from Lemma \ref{upper-bound-morphisms}. The upper bound for $\hookrightarrow$ follows from the second item in Lemma \ref{upper-bound-embeddings}, enumerating all possible choices of $f$. The upper bound for $\simeq$ is proved in \cite[Theorem 2.3]{jeandel_hardness_2015}. 
 The lower bounds follow from Corollary \ref{prop:some-berger-properties}. 
\end{proof}

When the input ranges over effective subshifts in dimension $d\geq 1$, we have a stronger Rice theorem in the sense that only trivial problems are decidable. For $d\geq 2$ this is a consequence of  \cite[Theorem 5.1]{carrasco-vargas_rice_2025}, but we include the straightforward proof.
\begin{theorem}\label{rice-theorem-for-effective}
    Let $d\geq 1$. Every nontrivial property $\mathcal P$ of $\Z^d$-effective subshifts is undecidable. More precisely, $\mathcal P$ is $\Sigma_1^0$-hard if the empty subshift satisfies $\mathcal P$ and $\Pi_1^0$-hard otherwise.
\end{theorem}
\begin{proof}
    See Appendix \ref{appendix-rice-theorem-for-effective}.
\end{proof}

\begin{corollary}
    Let $d\geq 1$ and let $Y$ be a $\Z^d$-subshift. The following problems with effective inputs are either trivial, $\Sigma_1^0$-hard, or $\Pi_1^0$-hard.
\begin{tasks}[style=itemize](6)
    \task $X=Y$
    \task $X\simeq Y$ 
    \task $X\subseteq Y$
    \task $X\hookrightarrow Y$
    \task $Y\subseteq X$
    \task $Y\hookrightarrow X$
\end{tasks}
\end{corollary}
This lower bound is not optimal in general, and we are interested in understanding the exact complexity of these problems and how it depends on the parameter $Y$. An upper bound is the complexity of the two-input versions of these problems with effective inputs:
\begin{theorem}[\cite{jeandel_hardness_2015}]\label{upper-bound-inclusion-two-sofic-subshifts}
Let $d\geq 1$. The problem $X\subseteq Y$ (resp. $X=Y$) with two effective $\Z^d$-subshifts as input is $\Pi_2^0$-complete. 
\end{theorem}
\begin{theorem}[\cite{jeandel_hardness_2015}]\label{embedding-two-effective}\label{conjugacy-two-effective-subshifts}
Let $d\geq 1$. The problem $X\hookrightarrow Y$ (resp. $X\simeq Y$) with two effective $\Z^d$-subshifts as input is $\Sigma_3^0$-complete. 
\end{theorem}

\section{The inclusion problem $\mathbf{Y\subseteq X}$}\label{sec:Y-contained-in-X}

As we mentioned before, Rice-like theorems cannot be applied to prove general lower bounds for $Y\subseteq X$ with SFT inputs, and in fact this problem can be decidable for some choices of $Y$. 
\begin{theorem}\label{characterization-complexity-inclusion}
    Let $Y$ be a 
    subshift. Then the problem $Y\subseteq X$ on SFT inputs is enumeration-equivalent to $\Le(Y)^c$; that is, there exists a computable function which maps any enumeration of $\Le(Y)^c$ to an enumeration of all SFTs $X$ such that $Y \subseteq X$, and conversely.
\end{theorem}
\begin{proof}
    A pattern $p$ is in $\Le(Y)^c$ if and only if $Y \subseteq X_{\{p\}}$. To an enumeration $E$ of $\Le(Y)^c$, we associate the enumeration of all SFTs $X_{\F}$ such that $\F \subseteq E$.

    Conversely, to an enumeration $E$ of SFTs $X$ such that $Y \subseteq X$, we associate the enumeration of all patterns $p$ such that $p \in \F$ for some $X_{\F} \in E$.
\end{proof}
This result generalizes \cite[Proposition 6.2]{carrasco-vargas_rice_2025} which proves that the problem $Y\subseteq X$ for SFT inputs is decidable exactly when $Y$ has computable language. 

In the next result we are able to classify the possible difficulties of this problem when $Y$ is a $\Z^d$-SFT, $d\geq 2$. Namely, this is the class of all computably enumerable Turing degrees (all Turing degrees of $\Sigma_1^0$ subsets of $\N$). 
\begin{theorem}\label{prop:alldegree}
The class of Turing degrees of languages of SFTs on dimension $d\geq 2$ equals the class of computably enumerable Turing degrees.
\end{theorem}
\begin{proof}
We fix $d=2$. Since the language of a $\Z^2$-SFT is always a $\Pi_1^0$ set, its Turing degree is a computably enumerable degree. To prove our claim, it suffices to show that for every $\Pi_1^0$ set $S\subseteq \N$ we can find a $\Z^2$-SFT $Y_S$ such that $\Le(Y_S)$ and $S$ are Turing-equivalent. Consider the subshift $X_S\subset\{0,1\}^\Z$ defined by the set of forbidden patterns $\{01^n0 : n\notin S\}$. Then $X_S$  is effective and $\Le(X_S)$ is Turing-equivalent to $S$. By \cite{aubrun2013simulation}, there is a $\Z^2$-SFT $Y_S$ that simulates $X_S$; without stating their result formally, we only stress that there is a computable surjection $\Le(X_S) \to \Le(Y_S)$ (projective subdynamics) and, as shown in \cite[Section 5]{barbieri2024zero}, there is a computable \emph{reconstruction function} that allows one to compute $\Le(Y_S)$ from $\Le(X_S)$, so that $\Le(Y_S)$ and $\Le(X_S)$ are Turing-equivalent.
\end{proof}

In the previous proofs we are not able to obtain a many-one equivalence, so we ask the following question.
\begin{question}
    Which many-one degrees are achieved by the problem $Y\subseteq X$ with SFT inputs?
\end{question}

For effective inputs, the upper bound is reached.
We are indebted to Mathieu Hoyrup (personal communication) for the following proof.

\begin{theorem}\label{thm:upper-bound-incl}
    There exists an SFT $Y$ such that the problem $Y\subseteq X$ for effective inputs is $\Pi^0_2$-complete.
\end{theorem}

\begin{proof}
    A number $l$ is left-c.e.\ if there exists a computable sequence $(l(k))_{k \in \N}$ of rational numbers such that $l(k)\le l(k+1) \le l$ for all $k$ and $l(k) \overset{k \rightarrow \infty}{\longrightarrow}l$. Chaitin's omega constant, denoted by $\Omega$, is left-c.e. Let $(l_i)_{i \in \N}$ be a computable enumeration of all left-c.e.\ numbers.
    
    First we prove that $B=\{i \ | \ \Omega \le l_i \}$ is $\Pi^0_2$-hard. We reduce the problem \textbf{Total}. Given $M$ a Turing machine, let $M'$ be the Turing machine which simulates $M(0)$, then $M(1)$, and so on. Let $k$ be such that $l_k(n)=\Omega(\#\{i\le n \ | \ M \ \mbox{ halts at time } i \mbox{ in the computation of } M' \})$
    It is clear that $M \in \textbf{Total} \Leftrightarrow k \in B$.
    
    Next we prove that $A=\{i \ | \ l_i \le \Omega\}$ is also $\Pi^0_2$-hard, reducing $B$. Given $i\in \N$, we can compute $c \in \N$ and $k$ such that $c\Omega-l_i =l_k$ (this follows from the fact that $\Omega$ is Solovay-complete \cite[Definition~3.2.28]{Nies}). Then we compute $n$ such that $l_n=\frac{l_k}{c-1}$. One has $\Omega \le l_i \Leftrightarrow l_n \le \Omega$.

     Finally, let $Y$ be the SFT which simulates a machine (as in Figure \ref{fig:computationtiles}) performing the following computation: given a rational number $q$ as input, compute successive approximations $\Omega(n)$ of $\Omega$ and checks that for all $n$, $\Omega(n) < q$, halting if this is not the case. Furthermore, the halting symbol is forbidden. Hence, an input $q$ is in $\Le(Y)$ if and only if it is a rational number such that $\Omega \le q$.

    We prove that $X=Y$ is $\Pi^0_2$-hard reducing the set $A$. For any $i \in \N$, let $X_i$ be the subshift of $Y$ that forbids finite inputs $q$ such that $q < l_i$. $X_i$ is effective since checking if $q<l_i$ is $\Sigma^0_1$. If $i \in A$, then every $q < l_i$ is not in $\Le(Y)$ and so $X_i=Y$. If $i \notin A$, then there exists $q$ such that $\Omega < q < l_i$, which may appear as input in $Y$ but is forbidden in $X_i$; hence $X_i \subsetneq Y$.
\end{proof}

\section{The embedding problem $\mathbf{Y\hookrightarrow X}$}\label{sec:Y-embeds-X}

We are mainly interested in the case of SFT inputs, as $Y \hookrightarrow X$ has a nontrivial boundary between the decidable and undecidable cases (with an upper bound of $\Sigma^0_1$ by Lemma \ref{upper-bound-embeddings}), while it is always undecidable ($\Pi^0_1$-hard) for effective inputs by Theorem \ref{rice-theorem-for-effective}. We are not able to characterize which properties of $Y$ make this problem decidable, but exhibit sufficient conditions for decidability and undecidability. 

We start with sufficient conditions for $Y \hookrightarrow X$ to be undecidable. 
\begin{theorem}\label{eventuallyperiodic}
Assume that $Y$ contains a configuration that is not strongly periodic, but is strongly periodic outside of a finite region (i.e. coincides with a strongly periodic configuration except on a finite set). Then $Y \hookrightarrow X$ for SFT inputs is $\Sigma^0_1$-hard.
\end{theorem}
\begin{proof}
We reduce the problem \textbf{Halt}. Let $M$ be a Turing machine given as input. Let $y_p \in Y$ be the configuration given by the hypothesis: there is a finite set $C$ such that $y_p$ is strongly periodic outside $C$ with period $p \in \A^{n \times n}$.

We define $X$ on the alphabet $\A'=\A \times T^{|\A|^{n^2}}$ where $T$ is the set of tiles corresponding to the Turing Machine $M$ as described in Figure \ref{fig:computationtiles}. The rules of adjacency of $T$ are slightly modified as follows. Given a configuration $x\in\A'^{\Z^2}$, its $i$-th \emph{layer} is its projection on the $i$-th coordinate. $X$ follows these rules:

\begin{itemize}	
    \item Patterns in $\A^{n \times n}$ are numbered $p_1, \dots, p_{|\A|^{n^2}}$.
	\item When a pattern $p_i\in \A^{n \times n}$ appears on the first layer with support $\llbracket a,a+n\llbracket\times \llbracket b,b+n\llbracket$, the layer $i+1$ must contain an initial tile (top right of Figure \ref{fig:computationtiles}) at position $(a,b)$, triggering the start of a computation of $M$ on empty input.
    The computation is halted if another occurrence of $p_i$ appears in the computing area.
\end{itemize}

Any configuration $y \in \A^{\Z^2}$ can be extended to a configuration $x_y \in X$ that coincides with $y$ on the first layer, by starting computations wherever required and, on each extra layer, leaving gray all cells with no computation. Any configuration of $X$ is obtained in this manner with possible additional computations with no initial tile (coming from infinity).

We first assume that $M$ halts, and prove that $Y \hookrightarrow X$. Define $f:Y\rightarrow X$ by $f(y)=x_y$ which is clearly injective. This function is local: if $M$ halts in $t$ steps, the function $f$ can determine which computation occurs on every layer from the value of $y$ in a $2t \times 2t$ square around each position. 

We now assume that $Y \hookrightarrow X$ and prove that $M$ halts. Let $f:Y\rightarrow X$ be local and injective. Because $f$ is local, $f(y_p)$ is strongly periodic of periods $(n,0)$ and $(0,n)$ when far away from $C$; and because $f$ is injective, $f(y_p)$ is not strongly periodic. Note that if the first layer of $f(y_p)$ was strongly periodic, then $f(y_p)$ would be either strongly periodic or not even strongly periodic outside of a finite region (because of computations coming from infinity); therefore the first layer of $f(y_p)$ cannot be strongly periodic. Therefore there is some pattern in the first layer that occurs in $C$ and not in the strongly periodic region. This pattern has a rightmost uppermost occurrence that triggers a computation which is never truncated. Since $f(y_p)$ is strongly periodic away from $C$, this computation must halt.
\end{proof}

It is remarkable that the previous result makes no use of any global property of $Y$, but only of the presence of a single configuration. However, this periodic structure is not needed, as the following example has no such configuation.

\begin{proposition}
    The problem $Rob \hookrightarrow X$ for SFT inputs is $\Sigma^0_1$-complete, where $Rob$ is the Robinson SFT without computation (see \cite{robinson_undecidability_1971}).
\end{proposition}

\begin{proof}
    Let $Rob^+$ be the Robinson SFT with a second layer on alphabet $\{\textrm{red}, \textrm{blue}\}$. Each computation macrotile is cut diagonally into a northwest red half and a southeast blue half; it is possible to do this in an SFT manner, by checking that each southwest and northeast corner are the ends of a red/blue diagonal border, and because the diagonal of a macrotile crosses sub-tiles along the diagonal. Clearly there is no morphism $Rob \to Rob^+$ since, for any $x\in Rob$ and $r\in \N$, you can find positions $i,j\in \Z^2$ such that $x_{i+\llbracket -r,r \rrbracket^2} = x_{j+\llbracket -r,r \rrbracket^2}$, but $i$ and $j$ are in the northwest, resp. southeast, halves of their macrotiles.

    We reduce the problem $\textbf{Halt}$. Given a Turing machine $M$ as input, let $X_M$ be the Robinson SFT in which we simulate $M$ as in \cite{robinson_undecidability_1971}, except that the computation is allowed to halt and the computation zone of a macrotile is entirely red if the computation halts in the macrotile, halved in a blue and a red half as in $Rob^+$ otherwise.
    
    Suppose that $M$ halts in $t$ steps. Define an embedding $Rob \hookrightarrow X_M$ that, to each $x\in Rob$, adds the first $t$ steps of computation and the corresponding red/blue layer, which can be done by a local function of exponential radius in $t$. Conversely, if $M$ does not halt, there is a morphism $X_M \to Rob^{+}$ that removes the computation, so $Rob$ cannot embed in $X_M$.
\end{proof}

 For the next result, we recall that a subshift is finite if and only if all configurations are strongly periodic \cite[Theorem 3.8]{ballier2008structural}.

\begin{proposition}\label{embedding-decidable-finite-subshifts}
Let $Y$ be a finite $\mathbb{Z}^d$-subshift ($d\geq 1$). Then $Y \hookrightarrow X$ for SFT inputs is decidable.
\end{proposition}
\begin{proof}
Since $Y$ is finite, we find a common set of linearly independent periods $v_1,\dots,v_d\in\mathbb{Z}^d$ for all configurations; that is, $\sigma^v(y)=y$ for every $v\in \{ v_1,\dots,v_d\}$ and all $y\in Y$. Given as input an SFT $X$, we add finitely many forbidden patterns to build the (finite) SFT 
\[X_0=\{x\in X : \sigma^{v_1}(x)=x,\dots,\sigma^{v_d}(x)=x\}.\]Then $Y$ embeds into $X$ if and only if $Y$ embeds into $X_0$ since the image of a configuration of period $v$ also has period $v$, and the last question is decidable as both systems are finite.
\end{proof}

We might expect that $Y \hookrightarrow X$ is always more difficult than the corresponding inclusion problem $Y \subseteq X$, that is, of $\Le(Y)$ (see Theorem \ref{characterization-complexity-inclusion}). However, this is not the case.

\begin{theorem}\label{answer1}
There exists an effective $\Z$-subshift $Y$ such that $Y \hookrightarrow X$ is decidable for SFT inputs but $\Le(Y)$ is not computable (and therefore $Y\subseteq X$ for SFT inputs is not decidable).
\end{theorem}

\begin{proof}
Let $E \subseteq \N$ be a $\Pi^0_1$-complete set. Let $Y$ be the $\Z$-subshift on alphabet $\A=\{0,1\}$ with forbidden patterns $\F=\{10^n1 : n\in\N\} \cup \{01^n0 \ : \ n \notin E\}$. $Y$ is effective since $E \in \Pi^0_1$, and $\Le(Y)$ is $\Pi^0_1$-complete. We claim that for any SFT $X$ on alphabet $\A_X$, $Y \hookrightarrow X$ if and only if there exists $a,b\in \A_X$ such that $a \neq b$ and two words $v, w \in \A_X^*$ such that the set ${}^{\omega}avb^*wa^{\omega}$ is included in $X$. This is decidable by building a labelled graph $A_X$ such that $x$ labels a walk in $A_X$ if and only if $x \in X$ (see \cite[Theorem 2.3.2]{lind_introduction_1995}), and checking if there are walks labeled ${}^{\omega}avb^*wa^{\omega}$ in $A_X$.

Let us prove the claim. Let $X$ be such that $Y \hookrightarrow X$ and $f: Y \rightarrow X$ be the corresponding embedding of radius $r$. Let $a=f(0^{2r+1})$ and $b=f(1^{2r+1})$ (where, by abuse of notation, we also denote by $f$ the local function). Since $f$ is injective, $a \neq b$. Consider configurations $x$ of the form ${}^{\omega}01^n0^{\omega}$ for $n \in E$. Since $x \in Y$, $f(x) \in X$ and $f(x)$ is of the form ${}^{\omega}av_nb^{n-2r}w_na^{\omega}$ for $n$ large enough, for some $v_n,w_n\in\A^{2r+1}$. Take $(v,w)$ such that $(v,w) = (v_n,w_n)$ for infinitely many $n$. The SFT $X$ can be seen as the set of infinite walks on $A_X$, so that ${}^{\omega}avb^*wa^{\omega} \in X$ by the pumping lemma.

For the other direction, let $a, b, v$ and $w$ which verify the property. Suppose without loss of generality that $|v|=|w|=r$. Define a local map $f:Y \rightarrow X$ of radius $r$ by: 
\[\forall p \in \Le_{2r+1}(Y),\quad f(p) = 
\begin{cases}
    a \text{ if } p=0^{2r+1}\\
    b \text{ if } p_0=1\\
    v(r+1-k)\text{ if }p \in 0^{r+k} 1\{0,1\}^{r-k}\text{ for }k>0\\
    w(k)\text{ if }p \in \{0,1\}^{r-k}10^{r+k}\text{ for }k>0\\
\end{cases}\]

We have $f(Y) \subseteq {}^{\omega}avb^*wa^{\omega} \subseteq X$, and $f$ is injective, so $f:Y \rightarrow X$ is an embedding.
\end{proof}
\begin{corollary}
    There is an effective $\Zd$-subshift $Y$ such that $Y\hookrightarrow X$ is decidable for SFT inputs but $\Le(Y)$ is not computable $(d\geq 1)$. 
\end{corollary}
\begin{proof}
This follows from Theorem \ref{answer1} and Proposition \ref{embedding-of-lift-sfts-is-decidable}.
\end{proof}
We do not know if a similar result applies when $Y$ is an SFT. Since our hardness result (Theorem \ref{eventuallyperiodic}) only applies to subshifts that are not minimal, it is natural to ask:

\begin{question}If $Y$ is an SFT such that $Y \hookrightarrow X$ is decidable for SFT inputs, is $\Le(Y)$ always computable?
\end{question}

The next result gives a positive answer in a restricted case. An example (from \cite{pavlov_structure_2023}) of SFT conjugate to a proper subsystem is the set of non-decreasing sequences in $\{0,1,2\}^\Z$.

\begin{proposition}
    If $Y$ is an SFT which is not conjugate to any proper subsystem, then $\Le^c(Y)$ is many-one reducible to $Y \hookrightarrow X$ for SFT inputs.
\end{proposition}  

\begin{proof}
    Let $\mathcal F$ be a finite defining set of forbidden patterns for $Y$. Then we can determine if a pattern $p$ belongs to $\Le(Y)^c$ by checking if $Y$ embeds in the SFT with same alphabet and set of forbidden patterns $\mathcal F\cup\{p\}$.
\end{proof}

\section{The inclusion $\mathbf{X\subseteq Y}$ and embedding $\mathbf{X\hookrightarrow Y}$ problems}\label{sec:X-contained-in-Y-and-X-embeds-Y} 
Our first result completely describes the complexity of $X \subseteq Y$ and $X \hookrightarrow Y$ when the parameter is an arbitrary SFT.

\begin{proposition}
    Let $Y$ be an SFT of dimension $d\geq 2$. Then the problems $X\subseteq Y$ and $X\hookrightarrow Y$ are $\Sigma_1^0$-complete for SFT/effective inputs. 
\end{proposition}

\begin{proof}
    The upper bound follows from Lemma \ref{upper-bound-inclusion}, and the lower bound follows from Corollary \ref{prop:some-berger-properties} and the fact that SFTs are in particular effective. 
\end{proof}
Therefore we shall be interested in taking a parameter subshift which is not an SFT. If the parameter $Y$ is effective then 
it follows from Theorem\ref{upper-bound-inclusion-two-sofic-subshifts} that the maximal complexity we can expect is $\Pi_2^0$ for $X\subseteq Y$ and $\Sigma_3^0$ for $X\hookrightarrow Y$. We show that these bounds are reached.

\begin{theorem}\label{inclusion-non-sft}
    Let $Y$ be an effective subshift which is not an SFT. Then the problem $X\subseteq Y$ on effective inputs is $\Pi_2^0$-complete.
\end{theorem}
\begin{proof}
    Let $\A$ be the alphabet for $Y$, and let $\mathcal F=\{p_1,\dots\}$ be a recursively enumerable set of forbidden patterns that defines $Y$. Since $Y$ is not an SFT by hypothesis, if we forbid finitely many patterns of $\mathcal F$, we obtain an SFT which properly contains $Y$. 

    We reduce the problem \textbf{Total}. Given a Turing machine $M$, we describe a new Turing machine  $M'$ as follows. On input $n\in\N$, $M'$ simulates $M$ on inputs $0,\dots,n$, and $M'$ halts if and only if $M$ halts on every such input. It is clear that $M\in \textbf{Total}$ if and only if $M'\in \textbf{Total}$. It is also clear that whenever $M\not\in\textbf{Total}$, the set $\{n\in\N : M'(n)\downarrow\}$ is finite. 

    Given a Turing machine $M$, consider the effective subshift $X_M$ on alphabet $\A$ defined by the forbidden patterns $\{p_n : M'(n)\downarrow\}$. When $M\in\textbf{Total}$ then $\{n\in\N : M'(n)\downarrow\}=\N$, so all patterns $p_n$ are forbidden and therefore $X_M=Y$. When $M\not\in\textbf{Total}$ then $\{n\in\N : M'(n)\downarrow\}$ is finite. This implies that $X_M$ is an SFT and in particular $Y\subsetneq X_M$. 
\end{proof}

The maximal complexity for inclusion can be achieved even for SFT inputs. Notice that the same problem is $\Pi^0_1$ in dimension $1$ where SFTs have computable language.

\begin{theorem}\label{thm:embed1}
In dimension $d\geq 2$, there exists an effective subshift $Y$ such that the problem $X\subseteq Y$ on SFT inputs is $\Pi^0_2$-complete.
\end{theorem}

\begin{proof}
Let $U$ be a universal Turing machine, that is, $U(M\#x)$ simulates the computation $M(x)$ on a potentially infinite input $x$.

Let $Z$ be a SFT which simulates $U$ as in Figure \ref{fig:computationtiles}: 
if some input on alphabet $\mathcal I = \{0,1,\#\}$ appears, then a computation of $U$ is simulated. The halting symbol of $U$ is forbidden. 
    
Let $Y\subseteq Z$ be the subshift where the input cannot be finite. $Y$ can be written as the union $Y_l \cup Y_r$, where $Y_l$ (resp. $Y_r$) is the SFT where the input must be infinite to the left (resp. right); in other words, the initial tile (resp. the pattern $a\square$) is forbidden for all $a\in \mathcal I$. $Y$ is effective as the union of two SFTs.

We reduce $X\subseteq Y$ to \textbf{Total}. Given a Turing machine $M$, let $X_M$ be the subshift of $Z$ where the input of $U$ is forced to be of the form $M\#x$ (with $x$ potentially infinite). $X_M$ is an SFT.
If $M \in \textbf{Total}$, then $\Le(X_M)$ does not contain any finite input so $X_M \subseteq Y$. Conversely, if $X_M \subseteq Y$ then $\Le(X_M)$ does not contain any finite input and so $M$ halts on every input.
\end{proof}

The next Theorem claims that the upper bound for $X \hookrightarrow Y$ on SFT inputs is reached.

\begin{theorem}
    In dimension $d \ge 2$, there exists an effective subshift $Y$ such that $X \hookrightarrow Y$ on SFT inputs is $\Sigma^0_3$-complete.
\end{theorem}

\begin{proof}
    Fix the injective function $o: n,m \mapsto 3^n 5^m$.
    Let $Y$ be an SFT that simulates a computation, as in Figure \ref{fig:computationtiles}, of a fixed Turing machine $U$. On an input of the form $M'\#n$ written in unary (except the $\#$ symbol), $U(M'\#n)$ computes $o(M',n)$ (where $M'$ is, by an abuse of notation, the code of the Turing machine), writes its value in unary with a special symbol, and then simulates $M'(n)$. The computation halts if the input is finite and malformed, and the halting symbol is forbidden. We distinguish input configurations (containing an input), finite-input configurations, and inputless configurations.
    All configurations have three other layers (black, red, green) each containing at most one vertical line, and finite-input configurations have one black, red and green line at positions $0, o(M,n)$ and $2o(M,n)$, respectively, from the start of the input. 
    The subshift $Y$ is effective; unicity of each line is enforced by "left" and "right" symbols, positions of the black and red line are enforced in an SFT manner (using the special symbol written by $U$ to find the position $o(M,n)$), and the only non-SFT constraint is to enforce the position of the green line. 

    We reduce the problem \textbf{COF}. Given a Turing machine $M$, let $X_M$ be as $Y$ except that the input is forced to start with $M\#$ and there is no green line. $X_M$ is an SFT.

    Suppose that $M \in \textbf{COF}$. There exists only finitely many finite-input configurations in $X_M$. We define an embedding $f:X_M \to Y$ which adds a green line as the symmetric to the black line relative to the red line at distance at most $2o(M,n)$, where $n$ is the longest halting input of $M$, and is the identity everywhere else.
    
    If $M \notin \textbf{COF}$, suppose for a contradiction that there exists an embedding $f:X_M \to Y$ of radius $r$. If $x \in X_M$ is an input configuration with no unneeded lines, then $x$ is $(0,1)$-periodic everywhere except on the upper-right quarterplane, so $f(x)$ is too. $f(x)$ cannot be $(0,1)$-periodic by injectivity; therefore $f$ maps input configurations to input configurations, and finite-input configurations to finite-input configurations. 

    Because $f$ is injective, we find arbitrarily large inputs $M\#n$ mapped by $f$ to $M'\#n'$ such that $o(M', n') \geq o(M,n)$, and in particular, we find $M$ and $n$ such that $2o(M', n') \geq o(M,n) + r$. Looking at the lower-right quarterplane, we find a monochromatic $2r \times 2r$ pattern mapped to a green symbol, so $f(x)$ has many green lines, which is a contradiction.
\end{proof}

The next result provides a natural sufficient condition for the complexity of these problems to reach their lower bounds.

\begin{proposition}\label{thm:embed-lower}
    Let $Y$ be an effective subshift that contains only finitely many SFT. Then the problems $X\subseteq Y$ and $X\hookrightarrow Y$ for SFT inputs are $\Sigma^0_1$-complete.
\end{proposition}

\begin{proof}
Since $Y$ contains finitely many SFT $Y_0,\dots Y_n$, it is enough to check whether $X = Y_k$ (respectively,  $X \simeq Y_k$) for some $k$ with $0\leq k\leq n$. These problems are $\Sigma^0_1$ by Lemma \ref{upper-bound-morphisms}.
\end{proof}

Proposition \ref{thm:embed-lower} applies, for example, to the $\Z$-subshift $Y_d$ on alphabet $\{0,1\}$ where all distances between consecutive ones must be distinct, which is far from minimal. Proposition \ref{thm:embed-lower} cannot hold for effective inputs, as the reader can check that $X\subseteq Y_d$ and $X\hookrightarrow Y_d$ are $\Pi^0_2$-hard for effective inputs.


\section{The equality problem $X=Y$}\label{sec:equality}

Corollary \ref{prop:some-berger-properties} shows that the equality problem for SFT inputs is $\Sigma_1^0$-complete as long as $Y$ is an SFT (otherwise, output ``no'' for every input). Therefore we are interested in the equality problem with effective inputs. We also assume that the parameter $Y$ is nonempty; determining whether an effective subshift is empty is a $\Sigma_1^0$-complete problem.  

Recall from Theorem \ref{upper-bound-inclusion-two-sofic-subshifts} that the equality problem with two effective inputs is $\Pi_2^0$-complete, so this is the general upper bound for the parametrized problem. In the next result we show that  $D(\Sigma_1^0)$ is always a lower bound.
\begin{theorem}\label{complexity-equality-problem}
    Let $Y$ be a nonempty effective subshift. Then the equality problem $X=Y$ for effective inputs is $D(\Sigma^0_1)$-hard.
\end{theorem}
\begin{proof} 
Let $\textbf{P}$ be the problem of determining whether an effective subshift on $\Z^d$ is empty ($d \geq 1$). It is known that $\textbf{P}$ is $\Sigma_1^0$-complete. This implies that $\textbf{P}\times \textbf{coP}$ is $D(\Sigma_1^0)$-complete. 

Given a nonempty effective subshift $Y$, we proceed by reduction of the problem $\textbf{P}\times \textbf{coP}$. Let $(Z_1,Z_2)$ be a pair of effective subshifts. Define $Z=(Z_2\times Y) \cup (Z_1\times \{\#\}^{\Z^d})$, where $\#$ is a fresh symbol, which is an effective subshift. Let $F\colon Z\to Y\cup\{\#\}^{\Z^d}$ be the projection to the second coordinate, and put $X=F(Z)$. Since effective subshifts are closed by factor maps \cite{aubrun_notion_2017}, $X$ is an effective subshift. This process is computable in the sense that one can computably enumerate a set of forbidden patterns for $X$ from $(Z_1,Z_2)$. Therefore:
    \[X=\begin{cases}
        Y \ &Z_1 =\emptyset, \ Z_2 \ne \emptyset \\
        \{\#\}^{\Z^d} \ &Z_1 \ne\emptyset, \ Z_2=\emptyset\\
        \emptyset \ &Z_1=\emptyset, \ Z_2=\emptyset \\
        Y \cup \{\#\}^{\Z^d} \ & Z_1\ne\emptyset, \ Z_2\ne\emptyset 
    \end{cases}\]
    In particular $X=Y$ if and only if $Z_1=\emptyset$ and $Z_2\ne\emptyset$. 
\end{proof}
The next result provides a sufficient condition for the equality problem to be $D(\Sigma_1^0)$-complete. 
\begin{proposition}\label{equality-for-sft-with-computable-language}
    Let $Y$ be a nonempty SFT  with computable language. Then $X=Y$ for effective inputs is $D(\Sigma^0_1)$-complete.
\end{proposition}
\begin{proof}
The lower bound  follows from Theorem \ref{complexity-equality-problem}, so we only need to prove the upper bound.  Since $Y$ is an SFT, the problem $X \subseteq Y$ is $\Sigma^0_1$ by Lemma \ref{upper-bound-inclusion}. Since $\Le(Y)$ is computable, $\Le(Y) \subseteq \Le(X)$ is a  $\Pi^0_1$ problem. We conclude that $X=Y$ is $D(\Sigma^0_1)$.
\end{proof}

\begin{proposition}
    There exists an SFT $Y$ such that $X=Y$ for effective inputs is $\Pi^0_2$-complete.
\end{proposition}

\begin{proof}
    Same proof as Theorem \ref{thm:upper-bound-incl}.
\end{proof}

\begin{proposition}\label{equality-non-sft}
    Let $Y$ be an effective subshift which is not an SFT. Then the equality problem $Y = X$ on effective inputs is $\Pi_2^0$-complete.
\end{proposition}

\begin{proof}
Same proof as Theorem \ref{inclusion-non-sft}.
\end{proof}

These results show that, in dimension $1$, the equality problem with effective inputs is never decidable, but its complexity completely characterizes SFTs within effective subshifts. 
\begin{corollary}
Let $Y$ be a nonempty effective $\Z$-subshift. The equality problem $Y=X$ with effective inputs is $D(\Sigma_1^0)$-complete if $Y$ is an SFT and $\Pi_2^0$-complete otherwise.
\end{corollary}
\begin{proof}
    Recall that in dimension $d=1$ every SFT has computable language \cite{lind_introduction_1995}. The claim follows from this, Theorem \ref{upper-bound-inclusion-two-sofic-subshifts}, Property \ref{equality-for-sft-with-computable-language}, and Property \ref{equality-non-sft}.
\end{proof}

\section{The conjugacy problem $X\simeq Y$}\label{sec:conjugacy}

As we saw in Proposition \ref{conjugacy-two-SFTs}, the conjugacy problem for SFT inputs is $\Sigma_1^0$-complete regardless of the parameter $Y$. Therefore we focus on effective inputs. In this case the complexity of the problem ranges between $D(\Sigma_1^0)$ and $\Sigma_3^0$ and strongly depends on the parameter.

Recall from Theorem \ref{conjugacy-two-effective-subshifts} that the conjugacy problem with two effective inputs is $\Sigma_3^0$-complete, so this is the general upper bound. We show that, even in $\Z$, there is an SFT whose conjugacy problem with effective inputs attains this maximal complexity. 
\begin{theorem}\label{conjugacy-hardness-Z}
Let $d\geq 1$. There is a $\Z^d$-SFT $Y$ such that $Y \simeq X$ for effective inputs is $\Sigma^0_3$-complete.
\end{theorem}
\begin{proof}
We prove the result for $d=1$. Let $Y\subset\{0,1,2\}^\Z$ be the SFT defined by forbidding $\{20,21\}$. Thus elements in $Y$ either belong to $\{0,1\}^\Z$, to $\{2\}^\Z$, or are a sequence of symbols in $\{0,1\}$ followed by infinitely many $2$'s. We prove that determining $Y\simeq X$ with effective inputs is $\Sigma_3^0$-complete using a reduction of the problem \textbf{COF}.

Given a Turing machine $M$ we compute below a Turing machine $M'$ such that:  
\begin{enumerate}
\item $M\in\textbf{COF} \Leftrightarrow M'\in \textbf{COF}$,
\item If $M\in\textbf{COF}$, then the cardinality of $\{n\in \N : M'(n)\uparrow\}$  is a power of two,
\item If $M\not\in \textbf{COF}$, then $\{n\in \N : M'(n)\uparrow\}$  contains no infinite arithmetic progression.
\end{enumerate}
Let $(S_n)_{n\in\N}$ be a computable injective enumeration of all the finite subsets of $\N$, $A=\{2^n : n\in\N\}$ (or any infinite computable set that does not contain $0$ nor any infinite arithmetic progression), and $g\colon\N\to A$ be the increasing computable bijection. Define $M'$ as follows:
\begin{itemize}
    \item if $n\not\in A$, then $M'(n)\downarrow$;
    \item if $n\in A$, then $M'(n)$ halts if and only if $M(k)\downarrow$ for some $k\in S_{g^{-1}(n)}$.
\end{itemize}

In this manner $\{n\in\N  : M'(n)\uparrow\}$ is a subset of $A$ and is in bijection with the finite subsets of $\{x\in \N : M(x)\uparrow\}$; in particular, if it is finite, its size is a power of two. Thus $M'$ satisfies the desired properties. 

Given a Turing machine $M$ we compute an effective subshift $X_M\subset\{0,1,2\}^\Z$ defined by the set of forbidden words $\{20,21\}\cup\{01^n2 : M'(n)\downarrow\}$. Every configuration in $X_M$ can be written as either $^\omega\{0,1\}^\omega$; $^\omega 2^\omega$; $^\omega 12^\omega$; or $^\omega\{0,1\}01^{n}2^\omega$ for some $n$ with $M'(n)\uparrow$. Note that, for all $n\geq 0$, $01^n2\in \Le(X_M)$ if and only if $M'(n)\uparrow$. 

We now prove that $M\in\textbf{COF}$ if and only if $X_M \simeq Y$.

First, assume that $M\not\in \textbf{COF}$. Since $M'\not\in\textbf{COF}$, we have $01^n2\in \Le(X_M)$ for infinitely many values of $n\in\N$. Assume that $X_M$ is an SFT to reach a contradiction. Since the language of an $\Z$-SFT is regular \cite{lind_introduction_1995}, we use the pumping lemma for regular languages and find an infinite arithmetic progression $P\subseteq \N$ such that $01^n2\in\Le(X_M)$ for all $n\in P$. This implies that $P\subset\{n\in\N:M'(n)\uparrow\}$, which is a contradiction with Property 2. Therefore $X_M$ is not an SFT and cannot be conjugate to $Y$.

Now assume that $M\in\textbf{COF}$. Then $\{n\in\N : M'(n)\uparrow\}$ is finite, and its cardinality is a power of two, say, $2^\ell$. Enumerate this set in increasing order as $\{n_1,\dots,n_{2^\ell}\}$. 
Let $\{w_1,\dots,w_{2^\ell}\}$ be the enumeration of all words in $\{0,1\}^\ell$ in lexicographic order. Define a map $F\colon X_M\to Y$ as follows: if $x\in X_M$ contains $01^{n_i}2$ for some $i=1,\dots,2^\ell$, then $F(x)$ is the same configuration where $01^{n_i}2^\ell$ is replaced by $w_i2^{1+n_i}$ (which has the same length); otherwise, $F(x)=x$. The map $F$ is well-defined because every $x\in X_M$ contains at most one occurrence of a word $01^{n_i}2$ for $i=1,\dots,2^\ell$, and it is immediate that $F$ is a conjugacy. 
\end{proof}
We now provide a lower bound. 
\begin{proposition}\label{hardness-conjugacy-with-effective-input}
    Let $Y$ be a nonempty effective subshift. Then the conjugacy problem $X\simeq Y$ with effective inputs is $D(\Sigma_1^0)$-hard. 
\end{proposition}
\begin{proof}
Let $A\subset\N$ be a $D(\Sigma_1^0)$-complete set, and $B$ and $C$ be $\Sigma_1^0$-sets such that $C\subseteq B$ and $A=B\smallsetminus C$. 
For each $n\in\N$, we define $X_n$ on alphabet $\{0,1\}$ as follows: if $n\in B$, we forbid the symbol $0$, and if $n\in C$, we forbid the symbol $1$. This set of forbidden patterns is recursively enumerable from $n$, so $(X_n)_{n\in \N}$ is a computable sequence of effective subshifts.
\[X_n=\begin{cases}
\{1\}^{\Z^d} &\ n\in B, \ n\notin C\\
\emptyset &\ n \in B, \  n\in C \\
\{0,1\}^{\mathbb{Z}^d} &\ n\notin B, \ n\notin C
\end{cases}\]
Notice that $Y\simeq Y\times X_n$ exactly when $|X_n|=1$ (as $Y\times X_n$ is otherwise empty or has too much entropy). This is equivalent to $n\in A$, which is a $D(\Sigma_1^0)$-complete problem. 
\end{proof}
We now examine dynamical conditions that may reduce the complexity of the problem. 
\begin{theorem} \label{abstract-complexity-of-conjugacy}
Let $Y$ be an SFT that is $P$-minimal for inclusion for a $\Pi^0_1$, monotonous, conjugacy-invariant SFT property $P$. In other words:
\begin{itemize}
\item If an SFT $X$ satisfies $P$ then any SFT conjugate to $X$ or containing $X$ satisfies $P$;
    \item $Y$ satisfies $P$, and any strict subshift $X \subsetneq Y$ does not satisfy $P$;
    \item There is an algorithm that, given a finite set of forbidden patterns, halts if and only if the associated SFT does not satisfy $P$.
\end{itemize}
Then the problem $X\simeq Y$ with effective inputs is $D(\Sigma_1^0)$.
\end{theorem}
\begin{proof}
Let $Y$ be an SFT as in the statement, and let $X$ be an effective subshift given as input. For all $n$, let $X_n$ be the SFT obtained by forbidding all the forbidden patterns of $X$ that we can computably enumerate in time $n$. We claim that $X$ and $Y$ are conjugate if and only if $(\exists n \ X_n \simeq Y) \land (\forall m, \ P(X_m))$.

Indeed, if $X$ and $Y$ are conjugate then $X$ is an SFT and so $\exists n \ X \simeq X_n$. Since $X\subseteq X_m$ for all $m$ and $P$ is monotonous, so $P(X_m)$ holds for all $m$.

Conversely, suppose that the formula above holds, and fix $n$ so that $X_n$ is conjugate to $Y$. Since $P$ is conjugacy-invariant, $X_n$ satisfies $P$ and every proper subshift of $X_n$ fails to satisfy $P$. If $X_n \neq X$, then $X_m$ is a proper subshift of $X_n$ for $m>n$ large enough, so $X_m$ fails to satisfy $P$, a contradiction with the condition. Thus $X = X_n \simeq Y$.
\end{proof}
\begin{remark}[\cite{amir_minimality_2025}]\label{conjfullsheasy}
    The following properties satisfy the conditions in Theorem \ref{abstract-complexity-of-conjugacy}.
    \begin{enumerate}
        \item Having a dense set of strongly periodic points. 
        \item Having computable topological entropy and being entropy-minimal. 
        \item Being minimal.
    \end{enumerate}
\end{remark}
Combining Remark \ref{conjfullsheasy} and Theorem \ref{abstract-complexity-of-conjugacy}, we find many examples of SFTs for which the conjugacy problem with effective inputs is $D(\Sigma_1^0)$-complete. For instance, this is the case for the fullshift on any number of symbols, any finite subshift, or any minimal subshift. 

We find a relationship between Theorem \ref{abstract-complexity-of-conjugacy} and results in \cite{amir_minimality_2025} on the notion of strong computable type for subshifts. 
An SFT $Y$ has computable language if and only if it is minimal for a $\Pi_1^0$ property (This is stated for a $\Sigma^0_2$ property in \cite[Corollary 13]{amir_minimality_2025}; we clarify this point in the appendix). This property is not guaranteed to be conjugacy invariant; in fact we provide an example where it cannot.
\begin{example}
The SFT $Y$ from Theorem \ref{conjugacy-hardness-Z} is minimal for a $\Pi^0_1$ property, but this property can not be taken conjugacy invariant. 
\end{example}
\begin{proof} Since $Y$ has computable language, it is minimal for a monotonous $\Pi_1^0$ property (as the property given in the proof of \cite[Corollary 13]{amir_minimality_2025} is monotonous). If this property could be taken to be conjugacy-invariant, then by Theorem \ref{abstract-complexity-of-conjugacy} the problem $X\simeq Y$ with effective inputs would be in $D(\Sigma_1^0)$. However, we proved in Theorem \ref{conjugacy-hardness-Z} that this problem is $\Sigma_3^0$-hard. 
\end{proof}
\begin{question}
    Which SFTs are minimal for a conjugacy-invariant $\Pi_1^0$ property? 
\end{question}

\section{Conjugacy and embedding in dimension $1$}\label{sec:dim1}

Given two $\Z$-SFTs as inputs, their equality and inclusion problems are easily seen to be decidable. In contrast, the decidability status of the conjugacy problem is a famous open question, and the embedding problem is probably very difficult as well (see Proposition \ref{embedding-zsft}). Therefore we cannot hope to find a $\Z$-SFT whose conjugacy or embedding problems for SFT inputs are undecidable. Still, obtaining decidability results for one-dimensional subshifts is useful, particularly for the $Y \hookrightarrow X$ problem, because they yield decidability results for higher-dimensional subshifts obtained by taking lifts, as the next result shows.

\begin{proposition}\label{embedding-of-lift-sfts-is-decidable}
    Let $Y'\subseteq \A^{\Z}$ be a $\Z$-subshift. Let $\pi\colon\mathbb{Z}^d\to \mathbb{Z}$ be any surjective group homomorphism, and consider the lift $Y=\{(y(\pi(g))_{g\in\mathbb{Z}^d} : y\in  Y'\}$.
    
    The problem $Y \hookrightarrow X$ with SFT (resp. effective) inputs is many-one equivalent to $Y' \hookrightarrow X$ with SFT (resp. effective) inputs.
    The same result holds for $Y \subseteq X$.
\end{proposition}

\begin{proof}
Given a subshift $X$, consider its periodic part $X_p = \{x\in X : gx=x \ \forall g\in \ker(\pi)\}$. $X_p$ is a subshift which is effective, resp. SFT, when $X$ is. Notice that $Y_p=Y$. Furthermore, $X_p$ is the lift of the subshift $W_X=\{y\in \A^{\Z} : (y(\pi(g))_{g\in\mathbb{Z}^d}\in X\}$
which is, again, effective/SFT when $X$ is.
Since the periods of a configuration are preserved by an embedding, it follows that $Y\hookrightarrow X$ if and only if $Y \hookrightarrow X_p$, which is equivalent to $Y'\hookrightarrow W_X$. A similar argument proves the other direction, and holds for $\subseteq$.
\end{proof}

\begin{remark}
    The previous proof fails for $X \hookrightarrow Y$ (and other problems) because $X\hookrightarrow Y$ if and only if $X_p \hookrightarrow Y$ and $X = X_p$; however, the second condition is undecidable in general.
\end{remark}

Lifts of $\Z$-SFTs with decidable embedding problem form a class of $\Z^d$-SFTs with decidable embedding problem. Unfortunately, we are not able to answer the following question.
\begin{question}\label{question:decidability-embedding-Z-SFTs}
    For which $\mathbb{Z}$-SFTs $Y$ is the problem $Y\hookrightarrow X$ with SFT inputs decidable?
\end{question}

One can see that this problem is decidable for NMC SFTs, defined in \cite{pavlov_structure_2023}. This is a large class in that it is dense in the space of $\Z$-subshifts for a suitable topology \cite[Theorem~3.6]{pavlov_structure_2023}.

We expect Question \ref{question:decidability-embedding-Z-SFTs} to be difficult, as it can be reduced to the (non-uniform) conjugacy problem for SFTs in some cases.
\begin{proposition}
\label{embedding-zsft}
Let $Y$ be a nonempty mixing $\Z$-SFT. The conjugacy problem $Y \simeq X$ with SFT inputs admits a many-one reduction to the embedding problem $Y \hookrightarrow X$ with SFT inputs. 
\end{proposition}

\begin{proof}
    Let $Y$ be a mixing $\Z$-SFT and let $X$ be an SFT given as input for the problem $X\simeq Y$. We check that $X$ is mixing, which is a decidable property (see \cite{lind_introduction_1995}), and answer negatively if it is not. If it is, then by Krieger's embedding theorem we have $Y\simeq X$ if and only if $h_{top}(Y)=h_{top}(X)$ and  $Y\hookrightarrow X$. Therefore we only need to check if $h_{top}(Y)=h_{top}(X)$, which is a decidable question, and finally check whether $Y\hookrightarrow X$. 
\end{proof}

\begin{remark}
    The same proof holds for $X \hookrightarrow Y$. We do not get a many-one equivalence since the answer to the problem $Y \hookrightarrow X$ is not immediate when $X$ is not mixing.
\end{remark}

To conclude, we provide one last decidability result which goes beyond the class of SFTs. Recall that the sunny side up is an effective (in fact sofic) subshift which is not SFT.
\begin{proposition}
    Let $Y = \{x\in \{0,1\}^{\mathbb{Z}} : |x^{-1}\{1\}|\leq 1\}$ be the sunny side up subshift.
    Then $Y \hookrightarrow X$ for SFT inputs is decidable.
\end{proposition}
\begin{proof}
    Check that $X$ contains one configuration of the form ${}^{\omega}awa^{\omega}$, where $a$ is a letter and $w$ a word not of the form $a^n$. This is decidable for $\Z$-SFTs by inspecting a representation (labeled finite graph whose infinite walks are configurations of the SFT, see \cite[\S 2.3]{lind_introduction_1995}). 
\end{proof}

\begin{remark}
Notice that the same problem is undecidable for the two-dimensional sunny side up by Theorem \ref{eventuallyperiodic}. More generally, we do not know of any subshift $Y$ in dimension $2$ with a decidable embedding problem $Y\hookrightarrow X$ with SFT inputs, except for lifts of $\mathbb{Z}$-SFTs or a union or slight generalization thereof.
\end{remark}

\section{Conclusion}

We conclude this article by a direction for future work. For a factor map $F$, the subshift $F(Y)$ is simpler than $Y$ from many points of view. It is natural to expect that a problem associated to $F(Y)$ is always easier than the same problem associated to $Y$. All our characterizations and examples are consistent with this belief, but we did not find a reduction or a systematic argument for arbitrary $Y$ and $F$. The same question can be asked for other dynamical operators such as total simulations, as introduced in \cite{lafitte_computability_2008}.

\bibliographystyle{plainurl}
\bibliography{references}

\newpage\appendix
\section{Proof of Theorem \ref{adyan-rabin-theorem-for-sfts}}\label{appendix-rice-theorems}
We now provide a proof for Theorem \ref{adyan-rabin-theorem-for-sfts}. The proof is almost the same proof given in \cite{carrasco-vargas_rice_2025}, but it has a subtle modification regarding the alphabet that ensures that  conjugacy invariance is not needed from the property.
\begin{proof}[Proof of Theorem \ref{adyan-rabin-theorem-for-sfts}]
    We describe a many-one reduction from the coDomino problem $\coDP$, which is $\Sigma_1^0$-complete according to Berger's theorem \cite{berger_undecidability_1966}. Let $X$ be an instance of $\coDP$, that is, a SFT for which we want to know whether it is empty. Let $X_{+}$ and $X_{-}$ be SFTs as in Definition \ref{def:berger}. Let $\A_{X_{+}}$, $\A_{X_{-}}$, and $\A_{X}$ be the alphabets for $X_{+}$, $X_{-}$, and $X$, respectively, which we assume are subsets of $\N$. Since the property is not assumed to be conjugacy invariant, we must be careful with the alphabets in our reductions. Let $N=|\A_{X}\times \A_{X_{-}}|$, let $M=\max\{a : a \in \mathcal \A_{X_{+}}\}$, and let $\B=\{M+1,\dots,M+N\}$. In this manner $\B$ has the same cardinality as $\A_{X}\times \A_{X_{-}}$, but we guarantee that $\A_{X_{+}}$ and $\B$ are disjoint. Let $Y$ be an SFT on alphabet $\B$ which is topologically conjugate to $X\times X_{-}$. Note that an SFT presentation for $Y$ can easily be computed from the input $X$. Finally, let 
    \[Z=X_{+}\cup Y.\]
    A presentation for $Z$ is easily computed from the input $X$. If $X$ is empty, then $Z$ is set-theoretically equal to $X_{+}$, and therefore it satisfies the property $\mathcal P$. If $X$ is nonempty, then there is a factor map from $Z$  to $X_{-}$, and therefore it does not satisfy the property $\mathcal P$. 
\end{proof}

\section{Proof of Theorem \ref{rice-theorem-for-effective}}
\label{appendix-rice-theorem-for-effective}
\begin{proof}[Proof of Theorem \ref{rice-theorem-for-effective}]
    Suppose that the empty subshift satisfies $\mathcal P$. We show that $\textbf{Halt}\leq_m \mathcal P$. Since the property is nontrivial, there is an effective subshift $X$ which fails to satisfy $\mathcal P$. Let $\A$ be its alphabet, and let $\mathcal F$ be a defining set of forbidden patterns. Let $M$ be a Turing machine, and let $X_M$ be the effective subshift on alphabet $\A$ with set of forbidden patterns $\mathcal F \cup \{p : p\in \A^{\{-n,\dots,n\}^d}\text{ and } M(\varepsilon)\downarrow\text{in $n$ steps}\}$. Then $X_M=\emptyset$ when $M(\varepsilon)\downarrow$, and $X_M=X_{\mathcal F}$ otherwise. 

    If the empty subshift does not satisfy $\mathcal P$, the same proof (taking $X$ satisfying $\mathcal{P}$) shows that $\textbf{coHalt}\leq_m \mathcal P$.
\end{proof}
\section{Proof of Lemma \ref{upper-bound-morphisms}}

As we mentioned before, we repeat ideas from \cite{jeandel_hardness_2015} in the next proof. We only provide the full proof because we do not assume $X$ to be SFT. 
\begin{proof}[Proof of Lemma \ref{upper-bound-morphisms}]
We start with the first item. We first assume that $X$ and $Y$ have the same alphabet. Let $\mathcal F=\{q_1,\dots,q_k\}$ be a defining set of forbidden patterns for $Y$. The fact that $Y$ is an SFT ensures that $X\subseteq Y$ if and only if no configuration from $x$ has a forbidden pattern from $\mathcal F$, which is equivalent to 
\[\mathcal F\subseteq \coLe(X).\]
Since $\coLe(X)$ is $\Sigma_1^0$ and $\mathcal F$ is finite, this is a $\Sigma_1^0$ condition. If $X$ and $Y$ have different alphabets $\A$ and $\B$ then we replace them by $\A\cup\B$ (this requires to computably modify the defining sets of forbidden patterns for $X$ and $Y$).  

Item 2 follows from Item 1, plus the fact that given an effective subshift $X$ and a local rule $f$ as in the statement, the image $F(X)$ is also effective and one can compute $\coLe(F(X))$ from $f$ and $\coLe(X)$. This follows from the well-known fact that effective subshifts on $\Z^d$ are closed by factor maps \cite[Proposition 2.7]{aubrun_notion_2017}.

We now prove the third item. Thanks to the second item, it suffices to show that the problem of determining whether $F$ is injective is $\Sigma_1^0$, provided that we have already verified that $F(X)\subseteq Y$. 

Suppose that $F(X)\subseteq Y$, and suppose without loss of generality that the domain of $f$ is $D=\{0,\dots,r-1\}^d$. We claim that $F$ fails to be  injective if and only if for all $n\geq r$ we can find two patterns $p$ and $q$ in $\Le_n(X)$ with $p(0_{\Z^d})\ne q(0_{\Z^d})$ and with $F(p)=F(q)$. This condition is clearly $\Pi_1^0$, so its negation is $\Sigma_1^0$. Suppose first that $F$ is not injective. Then we can find $x,y\in X$ with $x\ne y$ and  $F(x)=F(y)$. Shifting these elements if necessary, we can assume $x(0_{\Z^d})\ne y(0_{\Z^d})$. Then it suffices to choose $p$ and $q$ as the restrictions of $x$ and $y$ to $\{0,\dots,n-1\}^d$. Suppose now that for all $n\geq r$ we can find two patterns $p_n$ and $q_n$ in $\Le_n(X)$ with $p_n(0_{\Z^d})\ne q_n(0_{\Z^d})$ and with $F(p_n)=F(q_n)$. Then for each $n$ we choose a pair $(x_n,y_n)$ of elements in $X$ such that $p_n\sqsubseteq x_n$ and $q_n\sqsubseteq y_n$. Since $X\times X$ is compact and metric, we can find a subsequence of  $(x_n,y_n)$ which converges. Let $(x,y)$ be this limit. Our choices of $p_n$ and $q_n$ ensure that $x$ and $y$ differ at $0_{\Z^d}$ while $F(x)=F(y)$, so $F$ is not injective.
\end{proof}

\section{Technical point on \cite[Corollary 13]{amir_minimality_2025}}

Corollary 13 in \cite{amir_minimality_2025} states that a subshift has strong computable type -- in the case of SFT, this is equivalent to having computable language -- when it is minimal for a $\Sigma^0_2$ property. While not explicitely stated, this is equivalent to being minimal for a $\Pi^0_1$ property.

By definition, a $\Sigma^0_2$ property is a computable union of $\Pi^0_1$ properties. Therefore a $\Pi^0_1$ property is $\Sigma^0_2$. For the other direction, it is enough to see that being minimal for a union of properties is equivalent to being minimal for one of the properties.

\end{document}